\documentclass[12pt]{article}

\usepackage[margin=2.5cm, top=2.5cm, bottom=2.5cm]{geometry}

\usepackage[utf8]{inputenc} 
\usepackage[T1]{fontenc}    

\usepackage{url}            
\usepackage{booktabs}       
\usepackage{amsfonts}       
\usepackage{nicefrac}       
\usepackage{microtype}      
  
\usepackage{amsmath,amssymb,amsthm,xcolor}
\usepackage{array}
\usepackage{float}

\usepackage{wrapfig}
\usepackage[hidelinks,hypertexnames=false]{hyperref}
\usepackage[capitalise]{cleveref}
\theoremstyle{definition}

\theoremstyle{plain}
\newtheorem{theorem}{Theorem}
\newtheorem{lemma}{Lemma}

\newtheorem{proposition}{Proposition}
\newtheorem{fact}{Fact}
\usepackage{graphicx}
\usepackage{subcaption}
\usepackage{mdframed,framed}
\usepackage{lipsum}
\usepackage{tikz,calc}
\usepackage{enumitem}
\usetikzlibrary{calc}
\allowdisplaybreaks
\usepackage{pgfplots}
\usetikzlibrary{intersections, pgfplots.fillbetween}
\usepackage{epstopdf}
\usepackage{algorithmic}
\usepackage{booktabs}
\usepackage{changepage}
\usepackage{centernot}
\usepackage{stmaryrd}
\usepackage[nottoc]{tocbibind}
\newcommand{\forae}{%
  \tikz[baseline={(forall.base)}]{
    \node[inner sep=0pt, outer sep=0pt] (forall) {$\forall$};
    \fill[white] (-0.06em,0.05ex) rectangle (0.06em,0.25ex);
    \draw[line width=0.04em] (-0.06em, 0.7ex) -- (0.06em, 0.7ex);
  }
}

\usepackage{empheq}

\usepackage{shadethm}

\newshadetheorem{assume}{Assumption}
\definecolor{shadethmcolor}{HTML}{FFFFFF}
\definecolor{shaderulecolor}{HTML}{000000}
\newcommand{\R}{\mathbb{R}}

\newcommand{\cp}{\overline{\partial}}

\newcommand{\N}{\mathbb{N}}

\def\p{^}
\newcommand{\co}{\mathrm{co}\hspace*{.3mm}}

\usepackage{accents}

\title{\LARGE Reachability of gradient descent}

\begin{document}

\author{\large C\'edric Josz\thanks{\url{cj2638@columbia.edu,wo2205@columbia.edu}, IEOR, Columbia University, New York.} \and Wenqing Ouyang\footnotemark[\value{footnote}]}
\date{}

\maketitle

\begin{center}
    \textbf{Abstract}
    \end{center}
    \vspace*{-3mm}
 \begin{adjustwidth}{0.2in}{0.2in}
 ~~~We show that gradient descent can converge to any local minimum of a smooth semi-algebraic function. This holds if the step sizes are nonsummable and sufficiently small. The same results hold for the subgradient method on locally Lipschitz semi-algebraic functions if the step size is constant.
 
\end{adjustwidth} 
\vspace*{3mm}
\vspace*{3mm}



\section{Introduction}
\label{sec:Introduction}

Gradient descent seeks to minimize a function $f:\R\p n \to \R$ by choosing an initial point $x_0\in \R\p n$ and iterating
 $$\forall k \in \N, ~~~  x_{k+1} = x_k - \alpha_k \nabla f(x_k),$$
 where $\alpha_0,\alpha_1,\alpha_2,\hdots >0$ are the step sizes. As objective functions nowadays often have multiple local minima, the question arises as to which minima gradient descent selects. This is considered one of the keys to understanding generalization in deep learning \cite[Chapter 12]{bach2024learning}. 
 
 While large step sizes may preclude convergence to certain local minima \cite{nar2018step,wu2018sgd,cohen2021gradient,bolte2025convergence}, it is not known whether gradient descent can converge to any local minimum with sufficiently small step sizes. It can converge to isolated local minima \cite[Theorem 3.2]{absil2005convergence} of real analytic functions with sufficiently small step sizes, and to isolated saddle points \cite[Theorem 1]{boumal2024} of smooth functions with sufficiently small constant step sizes. Due to symmetry \cite{josz2025subdifferentiation}, local minima in most problems of interest nowadays are however not isolated (for e.g., matrix factorization, neural networks, etc.). Also, it is desirable to allow for variable step sizes in practice, for example, to converge to flat minima \cite{josz2023lyapunov}. 
 
 In this note, we show that for smooth functions that are semi-algebraic, or more generally definable in an o-minimal structure on the real field \cite{van1998tame}, the only requirements are that the step sizes be sufficiently small and nonsummable. As usual, $C^{1,1}_L$ denotes the set of continuously differentiable functions whose gradient is Lipschitz continuous with modulus $L$. Formally, we have:

\begin{theorem}
\label{thm:attain_discrete}
     Let $f:\R\p n \to \R$ be $C_L\p {1,1}$ definable and $\overline{x} \in \R\p n$ be a critical point of $f$ that is not a local maximum of $f$. Then
     $$\forall \epsilon>0, ~\exists \eta >0: ~\forall \{\alpha_k\}_{k\in \N}\subseteq (0,1/L],~ \sum_{k=0}\p {\infty} \alpha_k = \infty \implies \exists x_0 \in S_\eta(\overline{x}):~ B_\epsilon(\overline{x}) \ni x_k \to \overline{x}$$
where $\{x_k\}_{k\in\N}$ is defined by $x_{k+1} = x_k - \alpha_k \nabla f(x_k)$ for all $k\in\N$.
\end{theorem}

Our proof overcomes several hurdles compared with \cite[Theorem 1]{boumal2024}:
 \begin{enumerate}[label=\rm{(\rm{\roman*})}]
     \item In order to relax the isolatedness assumption, we harness the Lyapunov stability property of gradient descent inherited from the Kurdyka-\L{}ojasiewicz inequality \cite{kurdyka1998gradients}. When $\overline{x}$ is a local minimum, it reads
     $$\forall \epsilon>0,~\exists \delta,\overline{\alpha}>0:~ x_0\in B_\delta(\overline{x}) ~~\implies ~~ \{x_k\}_{k\in\N}\subseteq B_\epsilon(\overline{x})$$
     for some step sizes $\{\alpha_k\}_{k\in\N}\subseteq (0,\overline{\alpha}]$.
     \item In order to handle variable step sizes, we 
     use the curve selection lemma \cite[4.6]{van1996geometric} and the monotonicity lemma \cite[4.1]{van1996geometric}. The first step of the proof of \cref{thm:attain_discrete} already deviates from \cite[Theorem 1]{boumal2024}, as we need to carefully choose a sequence $\{a_i\}_{i\in \N}\subseteq\R\p n$ such that $$a_i\to \overline{x} ~~~\text{and}~~~ f(a_i)>f(\overline{x})$$
     instead of making an arbitrary choice. We choose it along an arbitrary continuous curve $\gamma:[0,\overline{t}]\to\R\p n$ such that
     $$\lim_{t\downarrow 0} \gamma(t) =\overline{x} ~~~\text{and}~~~ f(\gamma(t))>f(\overline{x}).$$
 \end{enumerate}

\cref{thm:attain_discrete} admits the following nonsmooth analogue.
\begin{theorem}
\label{thm:attain_nonsmooth}
    Let $f:\R\p n\to \R$ be locally Lipschitz definable and $\overline{x}\in \R\p n$ be Clarke critical. Then
    $$\exists\overline{\alpha}>0:\forall \alpha \in(0,\overline{\alpha}],\exists x_0\in\R\p n\setminus\{\overline{x}\}: x_k\to \overline{x}$$
    for any sequence $\{x_k\}_{k\in\N}\subseteq\R\p n$ such that $x_{k+1}=x_k-\alpha g_k$ for all $k\in\N$ and some $g_k\in \overline{\partial} f(x_k)$.
\end{theorem}

In continuous-time gradient dynamics, \cite[Theorem 3]{moussu1997dynamique} shows that one can converge to any critical point of a real analytic function that is not a local maximum (see also \cite[Proposition 3]{cris2024}). \cite[Section 4]{nowel2002trajectories} proves the same result when the function is simply smooth and the critical point is isolated. \cite{szafraniec2021stable} provides sufficient conditions to ensure the number of corresponding trajectories is infinite. For the sake of completeness, we will also analyze reachability in continuous-time subgradient dynamics for locally Lipschitz definable functions.

This note is organized as follows. The notations are introduced in \cref{sec:Notations}. \cref{sec:Main body} contains the main results. For pedagogical purposes, we begin with the simpler continuous-time dynamics in \cref{subsec:Continuous time}, then analyze its discrete-time counterpart in \cref{subsec:Discrete time}.

\section{Notations}
\label{sec:Notations}

As usual, $\mathbb{N} = \{0,1,\hdots\}$ and $\mathbb{R}_+ = [0,\infty)$.
Given two integers $a \leq  b$, we use the notation 
\begin{equation*}
    \llbracket a,b \rrbracket = \{a,a+1,\hdots,b\}.    
\end{equation*}
The Euclidean norm is denoted by $|\cdot|=\sqrt{\langle\cdot,\cdot\rangle}$. Also, $B_r(x)$ and $\overline{B}_r(x)$ denote the open ball and the closed ball of radius $r$ centered at $x$, respectively.
The sphere of radius $r$ centered at $x$ is denoted by $S_r(x)=\overline{B}_r(x)\setminus B_r(x)$. The indicator function of a set $C$ is denoted by $\delta_C$. The distance from a point $x\in \R^n$ to a set $A\subseteq \R^n$ is denoted by $$d(x,A)=\inf_{a\in A}|x-a| .$$ Let $\co A$ denote the convex hull of $A\subseteq \R\p n$. For a sequence $\{x_k\}_{k\in \mathbb N}$ and $I \subseteq\N$, we denote $\{x_i\}_{i\in I}$ by $x_I$. A set-valued mapping $F:\R^n\rightrightarrows \R^m$ is upper semicontinuous at $\overline{x}\in\R\p n$ if for any neighborhood $V$ of $F(\overline{x})$, there exists a neighborhood $U$ of $\overline{x}$ such that $F(U)\subseteq V$. 

Given a locally Lipschitz function $f:\mathbb{R}^n\rightarrow \mathbb{R}$, its Clarke subdifferential is defined as
$$\overline{\partial} f(\overline{x}) = \co \{ v \in \mathbb{R}^n : \exists x_k \xrightarrow[\Omega]{} \overline{x} ~\text{with}~ \nabla f(x_k) \rightarrow v\}$$
where $\Omega$ are the differentiable points of $f$ and the letter $\Omega$ under the arrow means $x_k \in \Omega$ and $x_k \rightarrow \overline{x}$.
For a $C^1$ function, its Clarke subdifferential and its gradient coincide \cite[Proposition 2.2.4]{clarke1990}. Moreover, the Clarke subdifferential of a locally Lipschitz function is locally bounded, compact convex valued and upper semicontinuous \cite[Proposition 2.1.2(a) and 2.1.5(d)]{clarke1990}. A point $x\in\R\p n$ is Clarke critical if $0\in \overline{\partial}f(x)$. A scalar $\ell$ is a Clarke critical value if there exists a Clarke critical point $x$ such that $f(x)=\ell$.
The set of Clarke critical values of a locally Lipschitz definable function is finite \cite[Corollary 9]{bolte2007clarke}.

A trajectory of $F:\R\p n \rightrightarrows\R\p n$ is an absolutely continuous curve $x:I\to\R\p n$ where $I$ is an interval of $\R$ such that
$$\forae t\in I, ~~~ x'(t) \in F(x(t)),$$
where $\forae$ stands for `almost everywhere'. We refer to solutions to the differential inclusion $\dot{x}\in F(x)$ as trajectories of $F$ over an interval $I = [0,T)$ for some $T\in(0,\infty]$ or $I = [0,T]$ for some $T\in(0,\infty)$. A solution $x:I\to\R\p n$ is maximal if for any other solution $y:J\to\R\p n$ such that $I\subseteq J$ and $x(t)=y(t)$ for all $t\in I$, we have $I=J$. A solution is globally defined if $I = [0,\infty)$.

\section{Main body}
\label{sec:Main body}

We begin with continuous-time dynamics, then consider discrete-time dynamics. We'll be using the following well-known fact, proved in the Appendix for convenience.

\begin{fact}
    \label{fact:length}
    Let $f:\R\p n \to \R$ be locally Lipschitz. The following hold: 
    \begin{enumerate}[label=\rm{(\rm{\roman*})}]
        \item If $x:[0,T)\to\R\p n$ is a maximal solution to 
        $$\forae t\in(0,T), ~~~ x'(t)\in -\overline{\partial}f(x(t)),$$
        such that 
        $\int_0\p T|x'(t)|dt<\infty$,
        then $T=\infty$ and $x(\cdot)$ converges to a Clarke critical point of $f$. \label{item:continuous}
        \item If $\{\alpha_k\}_{k\in\N}\subseteq(0,\infty)$ and $\{x_k\}_{k\in\N}\subseteq \R\p n$ are such that
        $$\forall k \in \N, ~ x_{k+1} = x_k - \alpha_k \overline{\partial} f(x_k),~~~ \sum_{k=0}\p \infty \alpha_k = \infty, ~~~\text{and}~~~\sum_{k=0}\p \infty |x_{k+1}-x_k|<\infty,$$
    then $x_k$ converges to a Clarke critical point of $f$. \label{item:discrete}
    \end{enumerate}
\end{fact}

\subsection{Continuous time}
\label{subsec:Continuous time}

The following stability result originates in \L{}ojasiewicz' work \cite{lojasiewicz1984trajectoires}. We do not claim any novelty as it is a simple generalization of \cite[Theorem 3]{absil2006stable} for real analytic functions. We propose a different proof however, using Kurdyka's length formula \cite[Theorem 2]{kurdyka1998gradients}.


\begin{lemma}
    \label{fact:stable}
Let $\overline{x}\in \R\p n$ be a local minimum of a locally Lipschitz definable function $f:\R\p n \to \R$. For all $\epsilon>0$, there exists $\delta>0$ such that, for all $x_0 \in B_\delta(\overline{x})$, there exists a solution to 
\begin{equation}
        \label{eq:DI}
    \left\{ \begin{array}{ccc}
        \dot{x} & \in & -\overline{\partial} f(x),  \\
        x(0) & = & x_0,
    \end{array} \right.
    \end{equation}
    and any solution to \eqref{eq:DI} is globally defined, takes values in $B_\epsilon(\overline{x})$, and converges to a local minimum of $f$.
\end{lemma}
\begin{proof}
    Let $\epsilon>0$ such that $f(x) \geq f(\overline{x})$ for all $x\in \overline{B}_\epsilon(\overline{x})$ and $f(\overline{x})$ is the sole Clarke critical value of $f$ reached in $\overline{B}_\epsilon(\overline{x})$, the set of which is finite by the definable Morse-Sard theorem \cite[Corollary 9]{bolte2007clarke}. Since $f$ is locally Lipschitz, by \cite[Proposition 1]{josz2023global} the differential inclusion \eqref{eq:DI} admits a local solution for all $x_0 \in B_\epsilon(\overline{x})$. By \cite[Proposition 7]{josz2023global} (see also \cite[Theorem 2]{kurdyka1998gradients}), there exists a concave definable diffeomorphism $\psi:\mathbb{R}_+ \rightarrow \mathbb{R}_+$ such that
    \begin{equation*}
        \int_0^T |x'(t)|dt \leq \psi(f(x(0))-f(x(T)))
    \end{equation*}
    for any solution $x:[0,T]\rightarrow B_\epsilon(\overline{x})$ to \eqref{eq:DI} with $T>0$. Since $f$ is Lipschitz continuous and definable, by the chain rule \cite[Corollary 5.4]{drusvyatskiy2015curves} (see also \cite[Lemma 5.2]{davis2020stochastic}) we have
\begin{equation*}
\forae t \in (0,T),~\forall v\in \overline{\partial} f(x(t)), ~~~ (f\circ x)'(t) = \langle v , x'(t) \rangle = -|x'(t)|^2.
\end{equation*} 
By continuity, there exists $\delta \in (0,\epsilon)$ such that $\sup \{ f(x) - f(\overline{x}) : x \in B_\delta(\overline{x})\} \leq \psi^{-1}(\epsilon - \delta)$. For any solution to \eqref{eq:DI} initialized in $B_\delta(\overline{x})$, if $T=\inf\{t \geq 0: x(t) \notin B_\epsilon(\overline{x})\}<\infty$, then 
    \begin{equation*}
        \epsilon - \delta < |x(T)-x(0)| \leq \int_0^T|x'(t)|dt \leq \psi(f(x(0))-f(x(T))) \leq \psi(f(x(0))-f(\overline{x})) \leq \epsilon - \delta
    \end{equation*}
    as $f(\overline{x}) \leq f(x(t)) \leq f(x(0))$ for all $t\in [0,T]$. Hence $T=\infty$, so that any solution to \eqref{eq:DI} initialized in $B_\delta(\overline{x})$ satisfies
    $$\forall t\geq 0, ~ x(t)\in B_\epsilon(\overline{x}) ~~~\text{and}~~~ \int_0\p \infty|x'(t)|dt \leq \psi(f(x(0))-f(\overline{x})).$$
    One now concludes by \cref{fact:length}.
\end{proof}

In order to analyze reachability, we propose to use the reverse differential inclusion.

\begin{lemma}
\label{lemma:reversedi}
    Let $f:\R^n\to \R$ be locally Lipschitz definable, $\overline{x}\in\R\p n$, and $\delta>0$ be such that $(f(\overline{x}),\infty)$ is devoid of critical values of $f$ reached in  $\overline{B}_\delta(\overline x)$. For all $x_0\in B_\delta(\overline{x})$, if $f(x_0)>f(\overline{x})$, then there exists a solution $x:[0,T]\to \R^n$ to the reverse differential inclusion
    \begin{equation}
        \label{eq:DI+}
    \left\{ \begin{array}{ccc}
        \dot{x} & \in & \overline{\partial} f(x),  \\
        x(0) & = & x_0.
    \end{array} \right.
    \end{equation}
    such that $x[0,T)\subseteq B_\delta(\overline{x})$ and $x(T)\in S_\delta(\overline{x})$.
\end{lemma}
\begin{proof}
    Since the Clarke subdifferential of a locally Lipschitz function is upper semicontinuous, locally bounded, and compact convex-valued \cite{clarke1990}, there exists an absolutely continuous solution $x:[0,T)\to \R^n$ such that either $T=\infty$ and $x([0,\infty))\subseteq B_\delta(\overline{x})$, or $T<\infty$, $x([0,T))\subseteq B_\delta(\overline{x}) 
    $, and $x(T)=\lim_{t\to T} x(t) \in S_\delta(\overline{x})$ by \cite[Chapter 2, Theorems 3 and 4]{aubin1984differential}. Alternatively, one can also use the same argument as in \cref{fact:stable} to show the existence of such a solution. In either cases, by the chain rule \cite[Corollary 5.4]{drusvyatskiy2015curves} (see also \cite[Lemma 5.2]{davis2020stochastic}) we have
\begin{equation*}
\forae t \in (0,T),~\forall v\in \overline{\partial} f(x(t)), ~~~ (f\circ x)'(t) = \langle v , x'(t) \rangle =  |x'(t)|^2.
\end{equation*} 
Let $$\zeta = \inf \{d(0,\overline{\partial}f(x)) : x \in B_\delta(\overline{x}), f(x) \geq f(x_0)\}.$$ It is positive since $f$ does not reach any Clarke critical values among $[f(x_0),\infty)$ in $\overline{B}_\delta(\overline{x})$. If it were zero, one would be reach a contradiction by upper semicontinuity of $\overline{\partial} f$. For any solution to \eqref{eq:DI+}, if $T=\infty$, then  
    $$f(x(t))-f(x(0)) \geq \int_0\p t |\overline{\partial} f(x(s))|\p 2 ds \geq \zeta\p2 t \to \infty,$$
    a contradiction. Therefore $T<\infty$, and hence our proof is done.
\end{proof}

We are now ready to state the reachability result in continuous-time.

\begin{proposition}
\label{prop:attain}
     Let $f:\R\p n \to \R$ be definable and locally Lipschitz. If $\overline{x} \in \R\p n$ is a critical point of $f$ but is not a local maximum of $f$, then there exists $x_0\neq \overline{x}$ such that the differential inclusion 
    \begin{equation}
        \label{eq:DI-}
    \left\{ \begin{array}{ccc}
        \dot{x} & \in & -\overline{\partial} f(x),  \\
        x(0) & = & x_0,
    \end{array} \right.
    \end{equation}
    has a global solution converging to $\overline{x}$.
\end{proposition}
\begin{proof}
    We first consider the special case where $\overline{x}$ is a local minimum of $f$. Let $\epsilon>0$ be such that $f(x) \geq f(\overline{x})$ for all $x \in B_\epsilon(\overline{x})$, and $f(\overline{x})$ is the sole Clarke critical value of $f$ reached in $\overline{B}_\epsilon(\overline{x})$. Let $\delta >0$ be given by \cref{fact:stable}. Since $\overline{x}$ is not a local maximum of $f$, there exists $B_\delta(\overline{x}) \ni a_k \to \overline{x}$ such that $f(a_k) > f(\overline{x})$. By \cref{lemma:reversedi}, the reverse differential inclusion starting from $a_k$ admits a solution $y_k:[0,T_k]\to \overline{B}_{\delta}(\overline{x})$ such that $y_k((0,T_k))\subseteq B_\delta(\overline{x})$ and $y(T_k)\in S_\delta(\overline{x})$. By taking a subsequence if necessary, we may assume $b_k\to b_\infty\in S_\delta(\overline{x})$. In addition, by \cref{fact:stable} the differential inclusion \eqref{eq:DI} admits a global solution $z_k(\cdot):\R_+ \to B_\epsilon(\overline{x})$ when initialized at $a_k \in B_\delta(\overline{x})$. Moreover, by concatenating $y_k(T_k-\cdot)$ and $z_k$, we can get a global solution $x_k:\R_+\to \overline{B}_{\delta}(\overline{x})$ to \eqref{eq:DI} starting at $b_k\in S_\delta(\overline{x})$.  
    
Let $L = \sup \{ |v|: v \in \overline{\partial} f(x), x \in B_\epsilon(\overline{x})\}$. Since $|x_k'(t)| \leq L$ for all $k\in \N$ and almost every $t> 0$, successively applying the Arzel\`a-Ascoli and the Banach-Alaoglu theorems (see \cite[Theorem 4 p. 13]{aubin1984differential}) yields a subsequence (again denoted $x_k(\cdot)$) and an absolutely continuous function $x:\R_+\rightarrow \mathbb{R}^n$ such that $x_k(\cdot)$ converges uniformly to $x(\cdot)$ on compact intervals $I$ and $x_k'(\cdot)$ converges weakly to $x'(\cdot)$ in $L^1(I,\mathbb{R}^n)$. Since $x_k(\cdot)$ is a solution to \eqref{eq:DI+} for every $k \in \N$, $x(\cdot)$ is a solution to \eqref{eq:DI} with $x(0) = b_\infty$ by \cite[Convergence Theorem p. 60]{aubin1984differential}. \cref{fact:stable} implies that $x(\infty) = \lim_{t\to \infty} x(t)$ exists and is a local minimum of $f$.

Let $\epsilon_k \downarrow 0$ and $t_k$ be such that $x_k(t_k) = a_k$. By \cref{fact:stable}, there exists $\delta_k>0$ such that any solution to \eqref{eq:DI} initialized in $B_{\delta_k}(\overline{x})$ (respectively $B_{\delta_k}(x(\infty))$) remains in $B_{\epsilon_k}(\overline{x})$ (resp. $B_{\epsilon_k}(x(\infty))$) and converges. After taking a subsequence if necessary, we have $a_k \in B_{\delta_k}(\overline{x})$. Thus $x_k(\infty) = \lim x_k(t) \in B_{\epsilon_k}(\overline{x})$ and $x_k(\infty) \to \overline{x}$. 

Let $\tau_k$ be such that $|x(\tau_k)-x(\infty)|\leq \delta_k/2$. By uniform convergence, for all $\ell$ large enough, we have $|x_\ell(t)-x(t)|\leq \delta_k/2$ for all $t\in[0,\tau_k]$. Thus 
$$|x_\ell(\tau_k)-x(\infty)|\leq|x_\ell(\tau_k)-x(\tau_k)|+|x(\tau_k)-x(\infty)|\leq \delta_k/2+\delta_k/2 = \delta_k,$$
i.e., $x_\ell(\tau_k)\in B_{\delta_k}(x(\infty))$,
and $x_\ell(\infty) \in B_{\epsilon_k}(x(\infty))$. Hence $x_k(\infty) \to x(\infty)$ after taking a subsequence if necessary. We conclude that $\overline{x} = x(\infty)$. The desired result follows by taking $x_0 = b_\infty \neq \overline{x}$.

We now consider the general case where $\overline{x}$ is not necessarily a local minimum of $f$. Observe that $\overline{x}$ is a local minimum of the definable locally Lipschitz function $g = \max \{f,f(\overline{x})\}$ but not a local maximum of it. By the previous special case, there exists $x_0\neq \overline{x}$ such that the differential inclusion 
    $$
    \left\{ \begin{array}{ccc}
        \dot{x} & \in & -\overline{\partial} g(x)  \\
        x(0) & = & x_0
    \end{array} \right.
    $$
    has a global solution converging to $\overline{x}$. Since 
    $g(x(a)) - g(x(b)) = \int_a\p b |x'(t)|\p 2 dt$,
    as soon as $g(x(t))$ reaches $g(\overline{x})$ from above, $x(t)$ stalls at $\overline{x}$ for all future time. As $g(x) = f(x)$ when $g(x) > f(\overline{x})$ and $0 \in \overline{\partial} f(\overline{x})$, $x(\cdot)$ is a solution to \eqref{eq:DI-}. 
\end{proof}

\subsection{Discrete time}
\label{subsec:Discrete time}


We begin with a discrete length formula. It is a simpler version of \cite[Proposition 8]{josz2023global} which we prove using a small subset of its arguments. Similar results can of course be found in \cite{absil2005convergence,attouch2013convergence}.


\begin{lemma}
\label{lemma:length_discrete}
Let $f:\mathbb{R}^n \rightarrow \mathbb{R}$ be $C\p{1,1}_L$, $\overline{x}\in \R\p n$, $f(\overline x)=0$, and $r>0$. Suppose $\psi:\R_+\to\R_+$ is a concave diffeomorphism such that $|\nabla (\psi\circ f)(x)|\geq 1$ for all $x\in B_r(\overline{x})\cap [f> 0]$. Then, for all $K \in \mathbb{N}$ and $\alpha_0,\hdots,\alpha_K>0$ such that $\sup_{k\in\llbracket 0,K\rrbracket} \alpha_k < 2/L$ and $(x_0,\hdots,x_{K+1}) \in (B_r(\overline{x}) \cap [f\geq 0])\p {K+1}\times [f\geq 0]$ such that $x_{k+1} = x_k - \alpha_k \nabla f(x_k)$ for all $k \in\llbracket 0,K\rrbracket$, we have
\begin{equation}
\label{eq:length_discrete}
   \sum_{k=0}^K |x_{k+1}-x_k| ~\leq~ \frac{\psi(f(x_0)-f(x_{K+1}))}{1-L\sup_{k\in\llbracket 0,K\rrbracket} \alpha_k/2}
\end{equation}
and $f(x_0)\geq\cdots \geq f(x_{K+1})\geq 0$.
\end{lemma}
\begin{proof}
Since $f$ is $C\p {1,1}_L$, by \cite[Lemma 1.2.4]{nesterov2003introductory}, for all $k\in\llbracket 0,K\rrbracket$ we have
\begin{align*}
    f(x_{k+1})-f(x_k) & \leq \langle \nabla f(x_k), x_{k+1}-x_k\rangle + \frac{L}{2}|x_{k+1}-x_k|^2 \\
    & = \left(\frac{L\alpha_k}{2}-1\right) |x_{k+1}-x_k| |\nabla f(x_k)|
\end{align*}
and thus
\begin{equation}
\label{eq:descent_lemma_1}
    |x_{k+1}-x_k| |\nabla f(x_k)| \leq \frac{2}{2- L\alpha_k}(f(x_k) - f(x_{k+1})).
\end{equation}
If $\nabla f(x_k)\neq 0$, then $f(x_k)>f(x_{k+1})\geq \cdots \geq f(x_{K+1})\geq 0$ and $1 \leq | \nabla(\psi \circ f)(x_k)| = \psi'(f(x_k)) |\nabla f(x_k)|$. Multiplying \eqref{eq:descent_lemma_1} by $\psi'(f(x_k))$ and using the concavity of $\psi$ yields 
$$ |x_{k+1}-x_k| \leq \frac{\psi'(f(x_k))(f(x_k) - f(x_{k+1}))}{1-L\alpha_k/2} \leq \frac{\psi (f(x_k)) - \psi (f(x_{k+1}))}{1-L\alpha_k/2}.$$
If $\nabla f(x_k) = 0$, then $|x_{k+1}-x_k|=0$ so the above equation still holds. We obtain the telescoping sum 
$$ \sum_{k=0}^K |x_{k+1}-x_k| \leq \frac{\psi(f(x_0)) - \psi(f(x_{K+1}))
}{1-L\sup_{k\in \llbracket 0,K\rrbracket}\alpha_k/2} \leq \frac{\psi(f(x_0)-f(x_{K+1}))
}{1-L\sup_{k\in \llbracket 0,K\rrbracket}\alpha_k/2}$$
where the second inequality again uses concavity \cite[Lemma 3.5]{josz2023convergence}.
\end{proof}

We need a refined stability result compared to continous-time because we cannot assume local optimality. Indeed, the argument at the end of the proof of \cref{prop:attain}, which converts saddle points into local minima, breaks smoothness.
\begin{lemma}
    \label{lemma:stable_discrete}
Let $f:\mathbb{R}^n \rightarrow \mathbb{R}$ be $C\p{1,1}_L$ definable and $\overline{x}\in \R\p n$. For all $\epsilon>0$, there exists $\delta>0$ such that, if $\{\alpha_k\}_{k\in \N}\subseteq (0,\infty)$ and $\{x_k\}_{k\in\N}\subseteq\R\p n$ are such that
\begin{equation}
\label{eq:alpha_x}
 \sup_{k\in \N} \alpha_k < 2/L,~~~ \sum_{k=0}\p {\infty} \alpha_k = \infty,~~~ \forall k\in \N, ~ x_{k+1} = x_k - \alpha_k \nabla f(x_k),
\end{equation}
then
\begin{equation}
\label{eq:delta_epsilon}
x_0 \in B_\delta(\overline{x})\cap [f\geq f(\overline{x})] ~~~\implies ~~~ \forall k\in\llbracket 1,\overline{k}\rrbracket\setminus\{\infty\},~~~ x_k\in B_\epsilon(\overline{x})
\end{equation}
where $\overline{k}=\inf\{k\in \mathbb{N}:f(x_{k+1})< f(\overline x)\}$. In particular, if $\overline{k}=\infty$, then $x_k\to x_\infty$ where $\nabla f(x_\infty)=0$ and $f(x_\infty)=f(\overline{x})$.
\end{lemma}
\begin{proof}
    Without loss of generality, $f(\overline{x})=0$. Since $f$ is $C\p{1,1}_L$ definable, there exist $\psi:\R_+\to\R_+$ and $r>0$ such that the statement of \cref{lemma:length_discrete} holds. 
    Let $\epsilon\in (0,r)$.
    By continuity of $f$, there exists $\delta \in (0,\epsilon)$ such that  $$\sup_{B_\delta(\overline{x})} f \leq \psi^{-1}((\epsilon - \delta)/2).$$
    If $K = \inf\{ k\in \N: x_{k+1} \notin B_\epsilon(\overline{x}) \land k+1\leq \overline{k}\}<\infty$, then $x_0,\hdots,x_K \in B_\epsilon(\overline{x})\subseteq B_r(\overline{x})$, $f(x_0),\hdots,f(x_{K+1})\geq 0$, and \cref{lemma:length_discrete} yields the contradiction
    $$
        \epsilon - \delta < |x_{K+1}-x_0| \leq \sum_{k=0}\p {K} |x_{k+1}-x_k| \leq 2 \psi(f(x_0)-f(x_{K+1})) \leq 2\psi(f(x_0)) \leq \epsilon -\delta.
    $$
    Hence $K=\infty$, so that $x_k\in B_\epsilon(\overline{x})\subseteq B_r(\overline{x})$ for all $k\in \N$ with $k\leq \overline{k}$. If $\overline{k}=\infty$, then  
    $$ \sum_{k=0}\p \infty |x_{k+1}-x_k| \leq 2\psi(f(x_0)) <\infty$$
    so $x_k\to x_\infty \in B_r(\overline{x})$ with $\nabla f(x_\infty)= 0$ by \cref{fact:length}. If $f(x_\infty)>0$, then $0=|\nabla (\psi \circ f)(x_\infty)| = \psi'(f(x_\infty))|\nabla f(x_\infty)|\geq 1$, a contradiction.
\end{proof}

In order to reverse the dynamics, as in the proof of \cref{prop:attain}, we will rely on the following standard fact (see \cite[Theorem 4.4]{poliquin1996prox} or \cite[Theorem 13.37]{rockafellar2009variational}).

\begin{fact}
\label{fact:prox}
    Given a $C_L\p {1,1}$ function $f:\R\p n \to \R$ and a scalar $\lambda \in (0,1/L)$, consider the proximal mapping $P_\lambda f:\R\p n\rightrightarrows\R\p n$ defined by
    $$P_\lambda f(x) = \arg\min_{y \in \R\p n} f(y) + \frac{1}{2\lambda}|y-x|\p 2.$$
    Then $P_\lambda f$ is single-valued, Lipschitz continuous, and for any $x\in \R^n$ and $x\p + = P_\lambda f(x)$, we have
    $$x\p + = x - \lambda \nabla f(x\p +),~~~ f(x) - f(x\p +) \geq
    \frac{\lambda}{2}|\nabla f(x\p +)|\p 2, ~~\text{and}~~~
         |x\p + -x| \leq \frac{2\lambda}{1-L\lambda} |\nabla f(x)|. $$
\end{fact}

We can now determine where gradient descent can converge to.

\begin{proof}[Proof of \cref{thm:attain_discrete}]
Since $\overline{x}$ is not a local maximum of $f$, $\overline{x}\in \overline{[f>f(\overline{x})]}$. By the curve selection lemma \cite[4.6]{van1996geometric}, there exists a definable continuous curve $\gamma:[0,\overline{t}]\to\R\p n$ with $\overline{t}>0$ such that $\lim_{t\downarrow 0} \gamma(t)=\overline{x}$ and $f(\gamma(t))>f(\overline{x})$. The monotonicity lemma \cite[4.1]{van1996geometric} implies that $f\circ \gamma$ is strictly increasing after possibly reducing $\overline{t}$.

Let $\epsilon>0$ be sufficiently small such that $f(\overline{x})$ is the sole critical value of $f$ reached in $\overline{B}_\epsilon(\overline{x})$. Let $\delta\in (0,\epsilon)$ be given by \cref{lemma:stable_discrete}, which ensures stability around $\overline{x}$. By continuity of $f$, there exists $\eta\in(0,\delta)$ such that
\begin{equation}
    \label{eq:eta}
    \max_{\overline{B}_\eta(\overline{x})} f < f(\gamma(\overline{t})).
\end{equation}
Let $\{\alpha_k\}_{k\in\N}\subseteq(0,1/L]$ be such that $\sum_{i=0}\p \infty \alpha_k = \infty$. We select $\{\alpha_k^i\}_{k\in \N}\subseteq (0,1/L)$ such that $\alpha_k^i\uparrow \alpha_k$ as $i\to\infty$ for all $k\in \N$, and $\sum_{k=0}^\infty \alpha_k^i=\infty$ for all $i\in \N$. Consider the family of functions $\{\omega_{i}\}_{i\in\N}$ defined by
    $$\omega_{i} = P_{\alpha_0^i}(-f)\circ \cdots \circ P_{\alpha_{i}^i}(-f).$$
Let $i\in\N$. By \cref{fact:prox}, for any $z\in \R\p n$, there exist unique points $x_0,\hdots,x_{i+1}\in \R\p n$ such that
    \begin{align*}
            \left\{
            \begin{array}{ccl}
                x_k &  =& \arg\max\limits_{y \in \R\p n} f(y) - |y-x_{k+1}|\p 2/(2\alpha_{k}^i), \\
                x_{i+1} & = & z,
            \end{array}
            \right.
    \end{align*}
    They satisfy
    $$\left\{
    \begin{array}{ccl}
        x_{k+1} & = & x_k - \alpha_k^i \nabla f(x_k), \\
        x_0 & = & \omega_i(z),
    \end{array}
    \right.$$
    $$f(x_k) - f(x_{k+1}) \geq
    \frac{\alpha_k^i}{2}|\nabla f(x_k)|\p 2,~~~\text{and}~~~
        |x_k -x_{k+1}| \leq \frac{2\alpha_k^i}{1-L\alpha_k^i}|\nabla f(x_{k+1})|.$$
In particular,
    \begin{equation}
         \label{eq:prox_descent}
         f(\omega_i(z)) - f(z) \geq \frac{1}{2}\sum_{k=0}^i  \alpha_k^i |\nabla f(x_{k})|^2.
    \end{equation}
Notice that $\omega_i(z) = z$ iff $\nabla f(z) = 0$ iff $f(\omega_i(z))=f(z)$.
By \eqref{eq:prox_descent} and \eqref{eq:eta}, we have $f(\omega_i(\gamma(\overline{t})))\geq f(\gamma(\overline{t}))>\max_{\overline{B}_\eta(\overline{x})} f$, so that $|\omega_i(\gamma(\overline{t}))-\overline{x}|>\eta$.
Also, $\nabla f(\overline{x})=0$ and thus $\omega_i(\overline{x})=\overline{x}$. Since $\omega_i$ is continuous, by the intermediate value theorem, there exists $t_i\in [0,\overline{t}]$ such that $|\omega_i(\gamma(t_i))-\overline{x}|=\eta$. 
Let $$a_i=\gamma(t_i)~~~\text{and}~~~b_i=\omega_i(a_i).$$ 
Up to a subsequence, $b_i\to b_\infty\in S_{\eta}(\overline{x})$ as $i\to \infty$. 
For all $i\in\N\cup\{\infty\}$, let $\{x\p i_k\}_{k\in \N}$ be such that
        \begin{align*}
            \left\{
            \begin{array}{ccl}
                x\p i_{k+1} & = & x\p i_k - \alpha_k^i \nabla f(x_k\p i), \\
                x\p i_0 & = & b_i,
            \end{array}
            \right.
        \end{align*}
        and consider the index 
$$\overline{k}_i=\inf\{k\in \mathbb N:f(x_{k+1}^i)< f(\overline x)\}.$$ 
By Lipschitz continuity of $\nabla f$ and $\alpha_k^i\to\alpha_k^\infty:=\alpha_k$, $\{x^i_k\}_{k\in\N}$ converges pointwise to $\{x_k^\infty\}_{k\in\mathbb N}$. Since $b_i=\omega_i(a_i)$, we have $x_{i+1}\p i = a_i$ and $$f(b_i) = f(x_0\p i) \geq f(x_1\p i)\geq \cdots \geq f(x_{i+1}\p i) = f(a_i) > f(\overline{x}).$$
In particular, $f(x_k\p i) \geq f(\overline{x})$ for all $k\in\llbracket 0,i+1\rrbracket$, $i+1\leq \overline{k}_i$, and stability yields $x_0\p i,\hdots,x_{i+1}\p i \in B_{\epsilon}(\overline{x})$. Passing to the limit for fixed $k$ yields $f(x_k\p \infty) \geq f(\overline{x})$, i.e., $\overline{k}_\infty = \infty$. Again by stability, $x^\infty_k\to x_\infty$, $\nabla f(x_\infty)=0$ and $f(x_\infty)=f(\overline x)$. 

We now show that $a_i\to\overline{x}$. This holds because $t_i\to 0$. Indeed, if not, then there is a subsequence (again denoted $\{t_i\}_{i\in \mathbb N}$) such that $t_i\geq \underline{t}$ for all $i\in\N$ for some $\underline{t}>0$. The constant
    $$\zeta = \inf \{|\nabla f(x)| : |x-\overline{x}|\leq \epsilon,~f(x)\geq f(\gamma(\underline{t}))\}$$
    is positive since $f(\overline{x})$ is the sole critical value of $f$ on $\overline{B}_{\epsilon}(\overline{x})$ and $f(\gamma(\underline{t}))>f(\overline{x})$. Together with $f(a_i) = f(\gamma(t_i))\geq f(\gamma(\underline{t}))$ and \eqref{eq:prox_descent}, this yields the contradiction 
    \begin{align*}
        f(b_i)-f(a_i) = f(\omega_i(a_i)) - f(a_i) \geq \frac{1}{2} \sum_{k=0}\p i \alpha_k|\nabla f(x_k\p i)|\p 2 \geq \frac{1}{2} \sum_{k=0}\p i \alpha_k^i\zeta\p 2 \to\infty,
    \end{align*}
where we have used $\alpha_k^i\uparrow \alpha_k$ and monotone convergence theorem to show that $$\sum_{k=0}^i\alpha_k^i\zeta^2\to \zeta^2\sum_{k=0}^\infty\alpha_k=\infty. $$
Given a sequence $\epsilon_j\downarrow 0$, by \cref{lemma:stable_discrete} there exists $\delta_j>0$ satisfying the stability property \eqref{eq:alpha_x}-\eqref{eq:delta_epsilon} around $x_\infty$. Fix $j\in\N$ and let $\tau_j$ be such that $|x^\infty_{\tau_j}-x_\infty|<\delta_j/2$. For all $i\in\N$ large enough, we have $|x^i_{\tau_j}-x_{\tau_j}^\infty|< \delta_j/2$ by pointwise convergence, and thus  
$$|x^i_{\tau_j}-x_\infty|\leq |x^i_{\tau_j}-x_{\tau_j}^\infty|+|x_{\tau_j}^\infty-x_\infty| <\delta_j/2+\delta_j/2=\delta_j.$$ 
In particular, when $\tau_j<i+1\leq \overline{k}_i$, by stability $|x_{i+1}\p i-x_\infty| =|a_i-x_\infty|<\epsilon_j$. Passing to the limit yields $|\overline{x}-x_\infty|<\epsilon_j$. As $j$ was arbitrary, $\overline{x} = x_\infty$, in other words, $x_k\p \infty\to \overline{x}$. The desired sequence is then $\{x_k\}_{k\in\mathbb N}=\{x_k^\infty\}_{k\in\mathbb N }$.
\end{proof}

    Note that the upper bound $1/L$ in \cref{thm:attain_discrete} cannot be improved. For example, let
    \[  f(x):=\begin{cases}
        \frac12 x^2   & \text{if }x\geq 0, \\
        -\frac{1}2{x^2} & \text{if } x<0.
    \end{cases}  \]
    Then $f\in C^1(\R)$ and $\nabla f$ is Lipschitz continuous with modulus $1$. However, gradient descent with constant step size $\alpha>1$ won't converge to $0$ unless starting at $0$.

     It would be desirable to extend \cref{thm:attain_discrete} to locally Lipschitz definable functions. However, that would require extending the stability result in \cref{lemma:stable_discrete}. This likely requires a more careful choice of step sizes, as suggested by the recent preprint \cite{lai2025diameter}. It guarantees convergence of bounded iterates of the subgradient method applied to Lipschitz continuous definable functions if the step sizes are $\alpha_k\approx 1/k$. Here, we present a proof for the reachability result of subgradient method with constant step size for Lipschitz definable functions. The case for variable step sizes is left for a future research direction.

The next lemma is a generalization of \cref{lemma:stable_discrete} with essentially the same proof.
\begin{lemma}
    \label{lemma:length_nonsmooth}
    Let $f:\mathbb{R}^n \rightarrow \mathbb{R}$ be locally Lipschitz definable and $\overline{x}\in \R\p n$. For all $\epsilon>0$, there exist $\delta,\rho>0$ such that, if $\{\alpha_k\}_{k\in \N}\subseteq (0,\infty)$ and $\{x_k\}_{k\in\N}\subseteq\R\p n$ are such that
$$\sum_{k=0}\p {\infty} \alpha_k = \infty,~~~ \forall k\in \N, ~~ x_{k+1} = x_k - \alpha_k g_k,~~g_k\in\cp f(x_k),~~f(x_{k+1})\leq f(x_k)-\rho \alpha_k|g_k|^2,$$
then
$$x_0 \in B_\delta(\overline{x})\cap [f\geq f(\overline{x})] ~~~\implies ~~~ \forall k\in\llbracket 1,\overline{k}\rrbracket\setminus\{\infty\},~~~ x_k\in B_\epsilon(\overline{x})$$
where $\overline{k}=\inf\{k\in \mathbb{N}:f(x_{k+1})< f(\overline x)\}$. In particular, if $\overline{k}=\infty$, then $x_k\to x_\infty$ where $0\in \cp f(x_{\infty})$ and $f(x_\infty)=f(\overline{x})$.
\end{lemma}

We are now ready to prove our second main result.

\begin{proof}[Proof of \cref{thm:attain_nonsmooth}]
   Without loss of generality, $\overline{x}=0$ and $f(0)=0$. Define $\Phi(t)=\mathrm{argmax}_{x\in \overline{B}_t(0)}f(x)$, $g(t)=\sup_{x\in \overline{B}_t(0)}f(x)$. Since $\Phi$ is definable, we know that there exists a definable choice of $\Phi$ such that $\phi(t)\in \Phi(t)$ for all $t\in \R$. Select $\delta>0$ such that $0$ is the sole Clarke critical value of $f$ on $\overline{B}_\delta(0)$. Then it is clear that $|\phi(t)|^2=t^2$ for all $t\in [0,\delta]$. Using monotonicity lemma, we may assume that $\phi\in C^2(0,\delta]$ by reducing $\delta$ if necessary. Using the first-order optimality condition, we see that 
    \begin{align*}
         0&\in \partial (-f+\delta_{\overline{B}_{t}(0)})(\phi(t))\overset{\rm (a)}{\subseteq} \partial(-f)(\phi(t))+  \partial \delta_{\overline B_{t}(0)}(\phi(t))\\
         &\overset{\rm (b)}{\subseteq}  \overline{\partial}(-f)(\phi(t)) +  \R_+ \phi(t) =-\overline{\partial}f(\phi(t)) +\R_+\phi(t),  
    \end{align*}
    where in (a) we have used the sum rule for subdifferential \cite[Corollary 10.9]{rockafellar2009variational}, in (b) we have used the fact that the subdifferential is contained in the Clarke subdifferential and \cite[Theorem 9.61]{rockafellar2009variational} to compute $\partial \delta_{\overline B_{t}(0)}(\phi(t))=\R_+\phi(t)$, and in (c) we have used the chain rule of Clarke subdifferential for scalar product \cite[Proposition 2.3.1]{clarke1990}. This implies that $A(t):=\{a\in \R_+:~a\phi(t)\in \overline{\partial}f(\phi(t))\}$ is nonempty for all $t\in [0,\delta]$. In particular, we can select a definable choice $a(t)\in A(t)$ such that $a:[0,\delta]\to \R_+$ is also definable. Using monotonicity lemma, we may assume that $a$ is monotone and continuous on $(0,\delta]$ by reducing $\delta$ if necessary. If $a(t)\to \infty$ as $t\to 0$, then $a$ must be monotonically decreasing on $(0,\delta]$. We select $\overline{\alpha}=1/(2a(\delta))$. In this case, by continuity,
    there exists some $t_0\in (0,\delta]$ such that $\alpha=1/t_0$, which means that $1/\alpha\phi(t_0)\in \partial \phi(t_0)$, and hence 
    \[    x_0=\phi(t_0),~x_1=0=\phi(t_0)-\alpha/\alpha\phi(t_0)\in x_0-\alpha\overline{\partial}f(x_0),~x_k=0~~\forall k\geq 1       \]
    is a Clarke subgradient trajectory of $f$ with step size $\alpha$. In the following we assume $\lim_{t\to 0}a(t)<\infty$, in which case $L=\sup_{t\in (0,\delta]}a(t) <\infty$. Noticing that $|\phi(t)|^2=t^2$, by differentiating this equality, we have $\langle \phi'(t),\phi(t) \rangle=t$ for all $t\in (0,\delta]$. Using \cite[Corollary 5.4]{drusvyatskiy2015curves}, we have 
    \begin{equation}
        \label{upper_1}
        \begin{aligned}
                f(\phi(t))&=\int_0^t\langle \overline{\partial}f(\phi(s)),\phi'(s) \rangle ds= \int_0^t a(s)\langle \phi(s),\phi'(s) \rangle ds \\
                &=\int_0^t a(s)sds \leq \frac{L}{2}t^2,~\forall t\in [0,\delta].
        \end{aligned}
    \end{equation}
    We take $\overline{\alpha}=1/(2L)$. Define the next set-valued mapping:
    \[     P_{\alpha}(x):=\mathrm{argmin}_{y} -\alpha f(y)+\delta_{\overline{B}_\delta(0)}(y)+\frac{1}{2} |y-x|^2.        \]
    Since the function $-\alpha f+\delta_{\overline{B}_\delta(0)}$ is lsc, proper and prox-bounded, we know that $P_\alpha$ is outer semicontinuous\footnote{For locally bounded set-valued mapping, this is equivalent to upper semicontinuity, e.g., see \cite[Theorem 5.19]{rockafellar2009variational}.} by \cite[Example 5.23(b)]{rockafellar2009variational}. Moreover, \eqref{upper_1} gives that $0=P_\alpha(0)$ for all $\alpha<1/L$. For any $x\in \overline{B}_\delta(0)$ and $y\in P_{\alpha}(x)$, we have 
    \[   -\frac{L\alpha}{2}|y^2|+\frac{1}{2}|y-x|^2  \leq -\alpha f(y)+\frac{1}{2} |y-x|^2\leq      \frac{1}{2}|x|^2,  \]
    which implies 
    \[\frac{1-L\alpha}{2}|y|^2\leq   \langle  x,y\rangle \leq |x||y| , \]
    and then $|y|\leq \frac{2}{1-L\alpha}|x|$. Therefore, we select $t_i\to 0$, and $a_i=\phi(t_i)$. Select $\eta\in (0,\delta)$ such that $\eta<\frac{1-L\overline{\alpha}}{2}\delta$. For each $i\in \mathbb N$, we define a sequence $x_0^i=a_i$ and $x^i_{k+1}\in P_\alpha(x_k^i)$, and $\ell_i=\inf\{k: x_{k}^i\notin B_{\eta}(0)\}$. First, we show that $\ell_i<\infty $ for all $i\in \mathbb N$. Define $\zeta_i=\inf\{ d(0,\overline{\partial} f(x)),~ x\in \overline{B}_\delta(0):~f(x)\geq f(\phi(t_i)) \}$, which is positive since $0$ is the sole Clarke critical value of $f$ on $\overline{B}_\delta(0)$. if $\ell_i=\infty $, then we have 
    \[    f(x_{k+1}^i)-f(x_0^i)\geq \sum_{j=0}^{k} \frac{1}{\alpha}|x_{j}^i-x_{j+1}^i|^2\geq \alpha\sum_{j=0}^k d(0,\partial f(x_{j+1}^i))^2 \geq (k+1)\alpha \zeta_i^2\to \infty.         \]
    Next, we argue that $\ell_i\to\infty$, this is clear since $|x_{k+1}^i|\leq \frac{2}{1-L\alpha}|x_k^i|$ and $b_i=x_0^i=\phi(t_i)\to 0$. Moreover, since $|x_{\ell_i-1}^i|<\eta$, we have $|b_i|=|x^i_{\ell_i}|\in [\eta, \frac{2\eta}{1-L\alpha})$. Reverse the order, we define 
    \begin{align*}
        \begin{cases}
            x_0^i=b_i,\\
            x_{k+1}^i= x_k^i-\alpha g_k^i,~~g_k^i\in \overline{\partial} f(x^k_i),~~\forall k\in \llbracket 0, \ell_{i}-1\rrbracket \\
            x_{k}^i=a_i=\phi(t_i),~~\forall k>\ell_i.
        \end{cases}
    \end{align*}
    Define $S_i=\{x_k^i\}_{k\in \mathbb N}\in (\overline{B}_\delta(0))^{\mathbb N}$. Using the Tychonoff theorem \cite[Theorem 37.1]{munkres2000topology}, we know that $(\overline{B}_\delta(0))^{\mathbb N}$ is a compact space under the product topology and is first-countable, so $S_i$ has a convergent subsequence. Passing to a subsequence if necessary, we may assume that $S_i\to S=\{x_k\}_{k\in \mathbb N}$ pointwisely. Using the upper semicontinuity of the Clarke subgradient, we know that for each $k\in \mathbb N$, there exists $g_k\in \overline{\partial} f(x_k)$ such that $x_{k+1}=x_k-\alpha g_k$. Moreover, we have $|x_0|\in [\eta, 2\eta/(1-L\alpha)]$. Moreover, from the inequality $f(x_{k+1}^i)\leq f(x_k^i)-\alpha|g_k^i|^2$, passing to the limit and using the continuity of $f$, we have $f(x_{k+1})\leq f(x_k)-\alpha|g_k|^2$. The rest of the proof is a repeated argument as in \cref{thm:attain_discrete} by using \cref{lemma:length_nonsmooth}.
\end{proof}
     \section*{Appendix}

\begin{proof}[Proof of \cref{fact:length}]
\ref{item:continuous}
Since $x(\cdot)$ is absolutely continuous, for all $t \in [0,T)$ we have
$$
    |x(t) - x(0)| = \left| \int_0^t x'(\tau)d\tau \right| \leq \int_0^t |x'(\tau)|d\tau \leq \int_0^{\infty} |x'(\tau)|d\tau.
$$
Hence $x(\cdot)$ is a bounded subgradient trajectory and $T=\infty$ by Proposition \cite[Proposition 2]{josz2023global}. Consider a sequence $t_0,t_1,t_2,\hdots \in \mathbb{R}$ converging to $\infty$. For all $N\in\mathbb{N}$, we have
\begin{align*}
    \sum_{k=0}^{N} |x(t_{k+1}) - x(t_k)| & = \sum_{k=0}^{N} \left|\int_{t_k}^{t_{k+1}} x'(t)dt \right| \\
    & \leq \sum_{k=0}^{N} \int_{t_k}^{t_{k+1}} \left| x'(t) \right| dt \\
    & = \int_{t_0}^{t_{N+1}} \left| x'(t) \right| dt \\
    & \leq \int_{t_0}^{\infty} \left| x'(t) \right| dt < \infty.
\end{align*}
Let $p$ and $q$ denote two integers such that $p \geq q \geq 0$. We have
    \begin{align*}
        |x(t_p) - x(t_q)| & = \left| \sum_{k=p}^q x(t_{k+1}) - x(t_k) \right| \\
        & \leq \sum_{k=p}^q |x(t_{k+1}) - x(t_k)| \\
        & \leq \sum_{k=p}^{\infty} |x(t_{k+1}) - x(t_k)| \xrightarrow{p\rightarrow \infty} 0.
    \end{align*}
Thus $x(t_0),x(t_1),x(t_2),\hdots$ is a Cauchy sequence and it must have a limit, say $l$. This limit is independent of the sequence $t_k$; indeed, if another sequence $t'_0,t'_1,t'_2,\hdots$ has a limit $l'$, then $t''_k := t_k$ if $k$ even and $t''_k := t'_k$ if $k$ odd converges towards $l'' = l = l'$. Assume that $x(\cdot)$ does not converge towards $l$. 
Then there exists $\epsilon >0$ such that, for all $\bar{t} \geq 0$, there exists $t \geq \bar{t}$ such that $| x(t) - l| > \epsilon$. 
We may thus recursively define a sequence $t_0,t_1,t_2,\hdots$ converging to $\infty$ such that $\|x(t_k) - l\| > \epsilon$. This yields a contradiction. 

It remains to show that the limit $l$ is a critical point of $f$. Since $x'(t) \in -\overline{\partial} f(x(t))$ for almost every $t > 0$, there exists a subset $S$ of $(0,\infty)$ of Lebesgue measure zero such that $x'(t) \in -\overline{\partial} f(x(t))$ for all $t \in (0,\infty)\setminus S$. If there exists $\epsilon>0$ and $\bar{t} \geq 0$ such that $|x'(t)| \geq \epsilon$ for all $t\geq (\bar{t},\infty) \setminus S$, then $\int_{\bar{t}}^t |x'(\tau)|d\tau \geq \epsilon(t-\bar{t}) \rightarrow \infty$ as $t \rightarrow \infty$, resulting in a contradiction. Hence, for all $\epsilon > 0$ and for all $\bar{t} \geq 0$, there exists $t \in  (\bar{t},\infty) \setminus S$ such that $|x'(t)|< \epsilon$. We may thus recursively define a sequence $t_0,t_1,t_2,\ldots \in (0,\infty) \setminus S$ converging to $\infty$ such that $x'(t_k)$ converges to zero. Since $x(\cdot)$ converges to $l$, $x(t_0),x(t_1),\ldots$ converges to $l$. In addition, since $t_0,t_1,t_2, \ldots \in (0,\infty)\setminus S$, it holds that $x'(t_k) \in -\overline{\partial} f(x(t_k))$ for $k=0,1,2,\hdots$. By virtue of \cite[2.1.5 Proposition (b) p. 29]{clarke1990}, we may take the limit in the inclusion, that is to say, $0 \in \partial f(l)$.

\ref{item:discrete} $\{x_k\}_{k\in \mathbb{N}}$ is a Cauchy sequence because
$$
    |x_p - x_q| \leq \sum_{k=p}^{q-1} |x_{k+1} - x_k| \leq \sum_{k=p}^{\infty} |x_{k+1}-x_k| \xrightarrow{p\rightarrow \infty} 0.
$$
Hence it has a limit $\overline{v}$. Notice that $v_k := -(x_{k+1}-x_k)/\alpha_k \in \partial f(x_k)$. If eventually $\|s_k\|\geq \epsilon$ (say starting at index $k_0$) for some $\epsilon>0$, then
$$
    \infty = \sum_{k=k_0}^{\infty} \alpha_k = \sum_{k=k_0}^{\infty} \frac{|x_{k+1}-x_k|}{|v_k|} \leq \frac{1}{\epsilon} \sum_{k=k_0}^{\infty} |x_{k+1}-x_k| < \infty.
$$
Hence 0 is a limit point of $v_k$ and by \cite[2.1.5 Proposition (d) p. 29]{clarke1990} $0 \in \partial f(\overline{v})$.
\end{proof}

\begin{proof}[Proof of \cref{fact:prox}]
    Since $f$ is $C\p {1,1}_L$, by \cite[Lemma 1.2.4]{nesterov2003introductory} it satisfies
\begin{equation}
    \label{eq:taylor}
    \forall x,y \in \R\p n, ~~~ |f(y)-f(x)-\langle \nabla f(x) , y-x \rangle | \leq L |y-x|\p 2/2.
\end{equation}
The function $g=f+|\cdot-x|\p 2/(2\lambda)$ is strongly convex. Indeed, summing
$$f(y')\geq f(y) + \langle \nabla f(y),y'-y\rangle - \frac{L}{2}|y'-y|\p 2$$
and
\begin{align*}
    \frac{1}{2\lambda}|y'-x|\p 2 & = \frac{1}{2\lambda}|y-x|\p 2 + \frac{1}{\lambda}\langle y-x,y'-y\rangle + \frac{1}{2\lambda}|y'-y|\p 2 \\
    & \geq \frac{1}{2\lambda}|y-x|\p 2 + \frac{1}{\lambda}\langle y-x,y'-y\rangle + \frac{L}{2}|y'-y|\p 2
\end{align*}
yields
$$g_x(y')\geq g_x(y) + \langle \nabla g_x(y),y'-y\rangle + \left(\frac{1}{2\lambda}-\frac{L}{2}\right)|y'-y|\p 2.$$
Thus $g_x$ admits a unique minimizer $y$ characterized by $\nabla g_x(y) = 0$. Letting $y'$ denote the minimizer of $g_{x'}$, we have
$x= y+\lambda\nabla f(y)$ and $x'=y'+\lambda\nabla f(y')$, so that $|x-x'|=|y-y'+\lambda(\nabla f(y)-\nabla f(y'))|\geq |y-y'|-\lambda|\nabla f(y)-\nabla f(y')| \geq (1-L\lambda)|y-y'|.$ Hence $|y-y'|\leq |x-x'|/(1-L\lambda)$, i.e, $P_\lambda f$ is $(1-L\lambda)\p{-1}$-Lipschitz continuous.

Going back to \cref{eq:taylor}, we have
\begin{align*}
    f(y) - f(x) + \frac{1}{2\lambda}|y-x|\p 2 \geq & \langle \nabla f(x) , y-x \rangle + (\lambda\p {-1} -L)|y-x|\p 2/2 \\
    = & (1/2)(\lambda\p {-1} - L)( |y-x|\p 2 - 2(\lambda\p{-1}-L)\p {-1} \langle \nabla f(x) , y-x \rangle) \\
    = & (1/2)(\lambda\p {-1} - L)(|y - x - (\lambda\p{-1}-L)\p {-1}\nabla f(x)|\p 2 \\
    & -|(\lambda\p{-1}-L)\p {-1}\nabla f(x)|\p 2).
\end{align*}
Notice that taking $y=x$ cancels out the left hand side so $x\p + = P_\lambda f(x)$ satisfies
$0 \geq |x\p + - x - (\lambda\p{-1}-L)\p {-1}\nabla f(x)|\p 2 -  |(\lambda\p{-1}-L)\p {-1}\nabla f(x)|\p 2$,
that is to say 
$|x\p + - x| \leq 2 (\lambda\p{-1}-L)\p {-1}|\nabla f(x)|$.
The optimality of $x\p +$ yields $f(x\p +) + |x\p +-x|\p 2/(2\lambda) \leq f(x)$, while the first-order optimality condition reads $x\p + = x - \lambda \nabla f(x\p +)$.
\end{proof}

\paragraph{Funding} This work is funded by NSF EPCN grant 2023032.

\paragraph{Data availability} 
We declare that we have no associated data.
\section*{\textbf{Declarations}}
\paragraph{Conflict of interest} The authors declare that they have no Conflict of interest.

\section*{Acknowledgments}
We thank the reviewers and the associate editor for their valuable feedback. We also thank Nicolas Boumal for pointing us to his blog and related references. 

\bibliographystyle{abbrv}    
\bibliography{references}
\end{document}